\documentclass[12pt,reqno]{amsart}
\usepackage{graphicx}
\usepackage{amsmath}
\usepackage{amssymb}
\usepackage{amscd}
\usepackage{mathrsfs}
\usepackage[all,cmtip]{xy}
 \usepackage{color}
\usepackage{stmaryrd}

\newtheorem{theorem}{Theorem}

\newtheorem{theoremc}{Theorem}

\newtheorem{cor}[theoremc]{Corollary\!\!}
\newtheorem{lem}[theorem]{Lemma}

\newenvironment{Proof}[1]{\textbf{#1.}}{\qed \vspace{5pt}}

\renewcommand\1{{\bf 1}}
\renewcommand\a{\alpha}
\renewcommand\b{\beta}

\newcommand\com[1]{}
\newcommand\C{{\mathbb C}}
\newcommand\Cc{{\let\mathcal\mathscr\mathcal C}}

\newcommand\D{{\mathcal D}}
\newcommand\DD{\mathbb{D}}

\renewcommand\l{\lambda}
\newcommand\La{\Lambda}

\newcommand\oo{\omega}
\newcommand\op[1]{\mathop{\rm #1}\nolimits}
\newcommand\ot{\otimes}
\newcommand\p{\partial}

\newcommand\R{{\mathbb R}}
\newcommand\Ss{\mathbb{S}^6}

\newcommand\vp{\varphi}
\newcommand\we{\wedge}

\newcommand\z{\sigma}

\begin{document}

 \title[Non-existence of orthogonal complex structures]{Non-existence of\\
 orthogonal complex structures on $\Ss$ \\ with a metric close to the round one}
 \author{Boris Kruglikov}
 \date{}

 \vspace{-15pt}
 \begin{abstract}
I review several proofs for non-existence of orthogonal complex structures
on the six-sphere, most notably by G.~Bor and L.~Hern\'andez-Lamoneda, but also by
K.~Sekigawa and L.~Vanhecke that we generalize for metrics close to the round one.
Invited talk at MAM-1 workshop, 27-30 March 2017, Marburg.
 \end{abstract}



 \maketitle

 \vspace{-18pt}

%%%%%%%%%%%%%%%%%%%%%%%%%%%%%%%%%%%%%%%%%%%%%%%%%%%%%%%%%%%%%%%%%%%%%%%%%%
\section{Introduction}\label{S1}

In 1987 LeBrun\footnote{As discussed in the paper by A.\,C.\ Ferreira in this volume, the same result
was established in 1953 by A.\,Blanchard \cite{B} using the ideas anticipating twistors.
The proofs of \cite{B,LeB} are reviewed there.}
\cite{LeB} proved the following restricted non-existence result for the 6-sphere.
Let $(M,g)$ be a connected oriented Riemannian manifold.
Denote by $\mathcal{J}_g(M)$ the space of almost complex structures $J$ on $M$ that
are compatible with the metric (i.e.\ $J^*g=g$) and with the orientation.
This is the space of sections of an $SO(6)/U(3)$ fiber bundle, so whenever non-empty it
is infinite-dimensional. Associating to $J\in\mathcal{J}_g(M)$ the almost symplectic structure
$\oo(X,Y)=g(JX,Y)$, $X,Y\in TM$, we get a bijection between $\mathcal{J}_g(M)$ and
the space of almost Hermitian triples $(g,J,\oo)$ on $M$ with fixed $g$.

 \begin{theorem}\label{thm1}
No $J\in\mathcal{J}_{g_0}(\Ss)$ is integrable (is a complex structure) for the standard (round) metric
$g_0$. In other words, there are no Hermitian structures on $\Ss$ associated to the metric $g_0$.
 \end{theorem}

There are several proofs of this statement, we are going to review some of those.
The method of proof of Theorem \ref{thm1} by Salamon \cite{Sal} uses the fact that the twistor space of $(\Ss,g_0)$ is
$\mathcal{Z}(\Ss)=SO(8)/U(4)$ which is a K\"ahler manifold (it has a complex structure
because $\Ss$ is conformally flat, and the metric is induced by $g_0$), and so the holomorphic embedding
$s_J:\Ss\to\mathcal{Z}(\Ss)$ would induce a K\"ahler structure on $\Ss$.
 % This to a large extent coincides with the idea of \cite{B}.

Here the symmetry of $g_0$ is used (homogeneity), so this proof is not applicable for $g\approx g_0$
(but as mentioned in \cite{BHL}, a modification of the original approach of \cite{LeB}, based on an
isometric embedding of $(\Ss,g)$ into a higher-dimensional Euclidean space, is possible).

A generalization of Theorem \ref{thm1} obtained in \cite{BHL} is as follows.

 \begin{theorem}\label{thm2}
Let $g$ be a Riemannian metric on $\Ss$. Denote by $R_g$ its Riemannian curvature,
considered as a $(3,1)$-tensor,
and by $\tilde{R}_g:\La^2T^*\Ss\to\La^2T^*\Ss$ the associated $(2,2)$ tensor (curvature operator).
Assume that its spectrum (15 functions $\lambda_i$ on $\Ss$ counted with multiplicities)
$\op{Sp}(\tilde{R}_g)=\{\lambda_{\op{min}}\leq\dots\leq\lambda_{\op{max}}\}$ is positive
$\lambda_{\op{min}}>0$ and satisfies $5\lambda_{\op{max}}<7\lambda_{\op{min}}$. Then
no $J\in\mathcal{J}_{g}(\Ss)$ is integrable.
 \end{theorem}

This theorem will be proven in Section \ref{S4} after we introduce the notations and recall
the required knowledge in Sections \ref{S2} and \ref{S3}. Then we will give another proof
of Theorem \ref{thm1} due to Sekigawa and Vanhecke \cite{SV} in Section \ref{S5}. Then in Section \ref{S6}
we generalize it in the spirit of Theorem \ref{thm2}. Section \ref{S7} will be a short
summary and an outlook.

\smallskip

Let us start with an alternative proof of Theorem \ref{thm1} following Bor and Hern\'andez-Lamoneda \cite{BHL}.

\smallskip

 \begin{Proof}{Sketch of the proof of Theorem \ref{thm1}}
Let $K=\La^{3,0}(\Ss)$ be the canonical line bundle of the hypothetical complex structure $J$.
Equip it with the Levi-Civita connection $\nabla$ that is induced from $\La^3_\C(\Ss)$ by the
orthogonal projection. The curvature of $K$ with respect to $\nabla$ is
 \begin{equation}\label{NF}
\Omega=R_\nabla|_{\Lambda^{3,0}}+\Phi^*\we\Phi =i\tilde{R}_g(\omega)+\Phi^*\we\Phi,
 \end{equation}
where $\Phi$ is the second fundamental form (see \S\ref{S31}). It has type $(1,0)$ and so $i\Phi^*\we\Phi\leq0$ (see \S\ref{S32}).
Since for the round metric $g=g_0$ we have $\tilde{R}_g=\op{Id}$, so
 $$
i\Omega=-\omega+ i\Phi^*\we\Phi<0.
 $$
Thus $-i\Omega$ is a non-degenerate (positive) scalar valued 2-form which is closed by
the Bianchi identity. This implies that $\Ss$ is symplectic which is impossible due to
$H^2_{\text{dR}}(\Ss)=0$.
 \end{Proof}

It is clear from the proof that for $g\approx g_0$ the operator $R_g\approx\op{Id}$ is still positive,
so the conclusion holds for a small ball around $g_0$ in $\Gamma(\odot^2_+T^*\Ss)$.
It only remains to justify the quantitative claim.

%%%%%%%%%%%%%%%%%%%%%%%%%%%%%%%%%%%%%%%%%%%%%%%%%%%%%%%%%%%%%%%%%%%%%%%%%%
\section{Background I: connections on Hermitian bundles}\label{S2}

Let $M$ be a complex $n$-dimensional manifold. % of $\dim_\C=n$.
In this section we collect the facts about calculus on $M$ important for the proof.
A hurried reader should proceed to the next section returning here for reference.

Let $\pi:E\to M$ be a Hermitian vector bundle, that is a holomorphic bundle over $M$
equipped with the Riemannian structure $\langle,\rangle$ in fibers for which the
complex structure $J$ in the fibers is orthogonal. Examples are the tangent bundle $TM$ and
the canonical line bundle $K=\Lambda^{n,0}(M)$.

Note that a Hermitian structure is given via
a $\C$-bilinear symmetric product $\odot^2(E\otimes\C)\to\C$ as follows: the restriction
$(,):E'\otimes E''\to\C$, where $E\otimes\C=E'\oplus E''=E_{(1,0)}\oplus E_{(0,1)}$
is the canonical decomposition into $+i$ and $-i$ eigenspaces of the operator $J$,
gives the Hermitian metric $\langle,\rangle:E\otimes E\to\C$, $\langle\xi,\eta\rangle=(\xi,\bar\eta)$.

There are several canonical connections on $E$.

\subsection{The Chern connection}
This is also referred to as the canonical metric \cite{GH} or characteristic \cite{GBNV} connection and
is constructed as follows. Recall that the Dolbeaux complex of a holomorphic vector bundle is
 $$
0\to\Gamma(E)\stackrel{\bar\partial}\to\Omega^{0,1}(M;E)=\Gamma(E)\otimes\Omega^{0,1}(M)\to
\Omega^{0,2}(M;E)\stackrel{\bar\partial}\to\dots
 $$
where the first Dolbeaux differential $\bar\partial$, generating all the other differentials
in the complex, is given by localization as follows (for simplicity of notations, everywhere below we keep
using $M$ for the localization).

If $e_1,\dots,e_m$ is a basis of holomorphic sections
and $\xi=\sum f^ie_i\in\Gamma(E)$ a general section, $f^i\in C^\infty(M,\C)$,
then $\bar\partial\xi=\sum\bar\partial(f^i)e_i$. It is easy to check by passing to another
holomorphic frame that this operator $\bar\partial$ is well-defined, and that its extension by
the Leibnitz rule $\bar\partial(\xi\otimes\alpha)=\bar\partial(\xi)\we\alpha+\xi\cdot d\alpha$
yields a complex, $\bar\partial^2=0$.

\begin{theorem}
There exists a unique linear connection on the vector bundle $E$, i.e.\ a map $D:\Gamma(E)\to \Omega^1(M,E)=\Omega^{1}(M)\otimes\Gamma(E)$, that is
 \begin{itemize}
\item compatible with the metric:
$d\langle\xi,\eta\rangle=\langle D\xi,\eta\rangle+\langle\xi,D\eta\rangle$,
\item compatible with the complex structure:
$D''=\bar{\partial}$. %$:\Gamma(E)\to \Omega^{0,1}(M;E)$.
 \end{itemize}
The first condition is $Dg=0$ and the second implies $DJ=0$.
\end{theorem}

Above, $D''$ is the $(0,1)$-part of $D$, i.e.\ the composition of $D$ with the projection $\Omega^1(M,E)\to\Omega^{0,1}(M,E)$.

\begin{proof}
The statement is local, so we can use a local holomorphic frame $e_i$ to compute.
Thus, a linear connection $D$ is given by a connection form
$\theta=[\theta_a^b]\in\Omega^1(M;gl(n,\C))$: $De_a=\theta_a^be_b$. We use the notations
$e_{\bar a}=\bar{e_a}$, $\theta_{\bar a}^{\bar b}=\bar{\theta_a^b}$, etc, cf.\ \cite{GH}.

Let $g_{a\bar b}=\langle e_a,e_b\rangle=(e_a,e_{\bar b})$ be the components of the Hermitian metric.
The first condition on $D$ writes
 $$
dg_{a\bar b}=\langle De_a,e_b\rangle+\langle e_a,D e_b\rangle=
\theta_a^c g_{c\bar{b}}+\theta_{\bar b}^{\bar c} g_{a\bar{c}}.
 $$
The second condition means that all $\theta_a^b$ are $(1,0)$-forms, so the above formula splits:
$\partial g_{a\bar{b}}=\theta_a^c g_{c\bar{b}}$ $\Leftrightarrow$
$\bar\partial g_{a\bar{b}}=\theta_{\bar b}^{\bar c} g_{a\bar{c}}$.

Consequently, the connection form satisfying the two conditions is uniquely given by
$\theta=g^{-1}\cdot\partial g$, or
in components $\theta_a^b=g^{b\bar{c}}\partial g_{a\bar{c}}$.
\end{proof}

We will denote the Chern connection, so obtained, by $\DD$.
In particular, there is a canonical connection $\DD$ on the tangent bundle of a Hermitian manifold.
Its torsion is equal to
 $$
T_\DD=\pi_{2,0}(d^c\omega)^\sharp,
 $$
where $d^c\oo(\xi,\eta,\zeta)=-d\omega(J\xi,J\eta,J\zeta)$,
$\sharp:\Lambda^3T^*M\hookrightarrow\Lambda^2T^*M\otimes T^*M\to\Lambda^2T^*M\otimes TM$ is the index lift operator
and $\pi_{2,0}:\Lambda^2T^*M\otimes TM\to\Lambda^{2,0}T^*M\otimes TM=\{B:B(J\xi,\eta)=B(\xi,J\eta)=JB(\xi,\eta)\}$ is the projection, cf.\ \cite{GBNV}.
This implies that the Chern connection $\DD$ on $TM$
has a non-trivial torsion unless $(g,J,\oo)$ is K\"ahler.

\subsection{The Levi-Civita connection}
A Hermitian metric induces a canonical
torsionless metric connection $\nabla$ on $TM$: $\nabla g=0$, $T_\nabla=0$.

Due to computation of the torsion $T_\DD$ above,
the Levi-Civita connection $\nabla$ does not preserve $J$ unless $M$ is K\"ahler.
In other words, $\DD=\nabla$ only in this case.

Choosing the frame $e_a=\p_{z^a}$, $e_{\bar{a}}=\p_{\bar{z}^a}$, for a holomorphic
coordinate system $(z^a)$ on $M$, we get
 $$
\nabla e_a=\Gamma^c_{ab}e^b\otimes e_c+ \Gamma^c_{a\bar{b}}e^{\bar{b}}\otimes e_c+
\Gamma^{\bar{c}}_{ab}e^b\otimes e_{\bar{c}}+\Gamma^{\bar{c}}_{a\bar{b}}e^{\bar{b}}\otimes e_{\bar{c}}
 $$
and (because $g_{ab}=0=g_{\bar{a}\bar{b}}$) the Christoffel coefficients have the standard but shorter
form, e.g.
 $$
\Gamma^c_{ab}=
\tfrac12g^{c\bar{d}}\Bigl(\frac{\p g_{a\bar{d}}}{\p z^b}+\frac{\p g_{b\bar{d}}}{\p z^a}\Bigr),\
\Gamma^c_{a\bar{b}}=
\tfrac12g^{c\bar{d}}\Bigl(\frac{\p g_{a\bar{d}}}{\p\bar{z}^b}-\frac{\p g_{a\bar{b}}}{\p\bar{z}^d}\Bigr),\
%\Gamma^c_{a\bar{b}}=
%\tfrac12g^{c\bar{d}}\Bigl(\frac{\p g_{a\bar{d}}}{\p z^{\bar{b}}}-\frac{\p g_{a\bar{b}}}{\p z^{\bar{d}}}\Bigr)
\text{etc.}
 $$
Introducing the 1-forms $\vartheta_a^c=\Gamma_{ab}^ce^b+\Gamma_{a\bar{b}}^ce^{\bar{b}}$
(not necessarily of $(1,0)$-type) we obtain the induced connection on the
holomorphic bundles $T_{(1,0)}M$, $T^{(1,0)}M$ (and their conjugate):
 $$
\nabla e_a=\vartheta_a^c e_c,\ \nabla e^c=-\vartheta_a^c e^a,\ \text{etc.}
 $$

\subsection{The canonical connection}
Though we will almost not use it, let us mention also the canonical connection \cite{GBNV,T}
 \begin{equation}\label{can}
\D=\tfrac12(\nabla-J\nabla J)=\nabla-\tfrac12J\nabla(J).
 \end{equation}
This connection is both metric $\D(g)=0$ and complex $\D(J)=0$.
The price of this additional (second) property is the emergence of torsion: $T_{\D}(X,Y)=\frac12(\nabla_X(J)JY-\nabla_Y(J)JX)$.

Clearly $\nabla$ is the canonical connection iff the structure $(g,J,\omega)$ is K\"ahler.
Also, if the Chern connection $\mathbb{D}$ is canonical, then $(\nabla_XJ)Y=(\nabla_YJ)X$,
and this implies that the structure $(g,J,\omega)$ is almost K\"ahler. A Hermitian
almost K\"ahler structure is necessarily K\"ahler \cite{GrH}.

\subsection{Induced connections}
The above connections naturally induce canonical connections on the canonical bundle $K$.

For the Chern connection $\DD e_a=\theta_a^ce_c$ this is given via the section
$\Omega=e^1\we\dots\we e^n\in\Gamma(K)$ by $\DD\Omega=-\op{tr}(\theta)\ot\Omega$,
where $\op{tr}(\theta)=\theta_a^a$. Due to
$\langle\Omega,\Omega\rangle=\op{det}(g^{a\bar{b}})=\det{g}^{-1}$ we also have
$\DD\Omega=\frac{-1}2\p\log g\ot\Omega$.

Similarly, for the Levi-Civita connection $\nabla e_a=\vartheta_a^ce_c$ we get
$\nabla\Omega=-\op{tr}(\vartheta)\ot\Omega$, and in general the connection form
$\op{tr}(\vartheta)=\vartheta_a^a$ differs from that for the Chern connection
(but coincides with it in the K\"ahler case).

\subsection{Curvature and the second fundamental form}
Pick a linear connection $D$ % (not necessary Hermitian)
on a vector bundle $E$ over $M$. Denote $\Omega^k(M,E)=\Gamma(\La^kT^*M\ot E)$.
Then $D$ can be uniquely extended to a sequence of maps
$D:\Omega^k(M,E)\to\Omega^{k+1}(M,E)$ by the Leibnitz super-rule:
for $\alpha\in\Omega^\bullet(M)$ and $s\in\Gamma(E)$ let
$D(\alpha\ot s)=d\alpha\ot s+(-1)^{|\alpha|}\a\we Ds$.

The curvature is the obstruction for $(\Omega^\bullet(M,E),D)$ to be a complex:
identify $D^2:\Omega^0(M,E)\to\Omega^2(M,E)$ with
$R_D\in\Gamma(\La^2T^*M\ot\op{End}(E))$,
$R_D(\xi,\eta)=[D_\xi,D_\eta]-D_{[\xi,\eta]}\in\Gamma(\op{End}(E))$,
$\xi,\eta\in\Gamma(TM)$.

Here it is important which sign convention we choose. In terms of the connection matrix
$\theta_E=(\theta_a^b)$ we get:
 $$
D^2e_a=D(\theta_a^be_b)=(d\theta_a^b-\theta_a^c\we\theta_c^b)e_b.
 $$
This is the Maurer-Cartan form of the curvature:
$\Theta_a^b=d\theta_a^b-\theta_a^c\we\theta_c^b$,
or in coordinate-free notation $\Theta_E=d\theta_E-\frac12[\theta_E,\theta_E]$.

Note that if $D$ is the Chern connection (on a Hermitian bundle), then $\Theta_E$ is a matrix of $(1,1)$-forms,
but for the Levi-Civita connection in general this is not the case.

Let $E_0\subset E$ be a holomorphic subbundle and $E_1=E_0^\perp$ its ortho-complement.
Since $E_0$ is a Hermitian bundle in its own, we have two first order differential operators
 $$
D_E|_{\Gamma(E_0)}:\Gamma(E_0)\to\Omega^1(M,E)\ \text{ and }\
D_{E_0}:\Gamma(E_0)\to\Omega^1(M,E_0).
 $$
The second fundamental form of the subbundle $E_0$ in $E$ with normal bundle $E_1$
is the tensor $\Phi\in\Omega^1(M)\ot\Gamma(\op{Hom}(E_0,E_1))$ given by
 $$
\Phi=D_E|_{\Gamma(E_0)}-D_{E_0}:\Gamma(E_0)\to\Omega^1(M,E_1).
 $$
Note that for $D=\DD$ the Chern connection, $\Phi\in\Omega^{(1,0)}(M,\op{Hom}(E_0,E_1))$.

The connection matrix in the splitting $E=E_0\oplus E_1$ writes
 $$
\theta_E=\begin{bmatrix}\theta_{E_0} & \Phi^*\\ \Phi & \theta_{E_1}\end{bmatrix},
 $$
where $\Phi^*=\bar\Phi^t$. Hence the curvature is
 $$
\Theta_E= d\theta_E-\theta_E\we\theta_E=
\begin{bmatrix}
d\theta_{E_0}-\theta_{E_0}\we\theta_{E_0}-\Phi^*\we\Phi & \star\ \\
\star & \star\
% d\theta_{E_1}-\theta_{E_1}\we\theta_{E_1}-\Phi\we\Phi^*
\end{bmatrix}
 $$
so that
 \begin{equation}\label{SFF}
\Theta_E|_{E_0}=\Theta_{E_0}-\Phi^*\we\Phi,
 \end{equation}
where for vector spaces $V,W$ and elements $\a,\b\in V^*\ot\op{End}(W)$, $X,Y\in V$, we let
$(\a\we\b)(X,Y)=\a(X)\b(Y)-\a(Y)\b(X)$, $\a^*(X)=(\a(X))^*$.

%%%%%%%%%%%%%%%%%%%%%%%%%%%%%%%%%%%%%%%%%%%%%%%%%%%%%%%%%%%%%%%%%%%%%%%%%%
\section{Background II: positivity}\label{S3}

In this section preliminary computations are made, following \cite{BHL}.
The first subsection is just linear algebra and so is applicable to a complex vector space
$(V,J)$ ($=T_xM$), equipped with a complex-valued symmetric $\C$-bilinear non-degenerate form $(\cdot,\cdot)$
with $(X,\bar X)\ge0$. The corresponding Hermitian metric is $\langle X,Y\rangle=(X,\bar{Y})$.

Let us call a 2-form $\omega\in\La^2V^*$ {\em positive\/} (resp.\ non-negative $\omega\ge0$)
if the corresponding bilinear form $b(X,Y)=\omega(X,JY)$ is symmetric positive definite (resp.\ positive semidefinite).
In other words, this 2-form is $J$-invariant, i.e.\ $\omega\in\La^{1,1}V^*$, and
$\frac1i\omega(X',X'')>0$ (resp.\ $\ge0$) for $X\neq0$.
Here $X'=\frac12(X-iJX)$ is the projection of $X\in V$ to $V_{1,0}$
and $X''=\bar{X'}=\frac12(X+iJX)$ is the projection of $X$ to $V_{0,1}$.

Next, a 2-form $\Omega\in\La^2V^*\otimes\op{End}(W)$
with values in Hermitian endomorphisms of a complex space $W$
is {\em positive\/} (resp.\ non-negative $\Omega\ge0$)
if the scalar valued 2-form $\langle \Omega w,w\rangle$ is positive (resp.\ $\ge0$) $\forall w\neq0$.

\subsection{Projection on the canonical bundle}\label{S31}
For an anti-symmetric endomorphism $A:V\to V$ denote $\hat{A}\in\La^2V^*$ the element given
by lowering indices: $\hat{A}(v\we w)=(v,Aw)=-(Av,w)$, $v,w\in V$.

Note that $(A^*\a,\b)=(\hat{A},\a\we\b)$ for arbitrary $\a,\b\in V^*$. Indeed, using
the operator $\sharp:V^*\to V$ of raising indices, we get
 $$
(\hat{A},\a\we\b)=\hat{A}(\a^\sharp\we\b^\sharp)=(\a^\sharp,A\b^\sharp)=\a(A\b^\sharp)=(A^*\a)(\b^\sharp)=(A^*\a,\b).
 $$
Here and below star denotes the usual pull-back $A^*:\La^kV^*\to\La^kV^*$.

 \begin{lem}
Let $n=\dim_\C V$ and $\omega(X,Y)=\langle JX,Y\rangle$ be the symplectic form on $V$. Then,
denoting $\pi_0:\La^nV^*\ot\C\to\La^{n,0}V^*$ the orthogonal projection, we have
 $$
\pi_0A^*\pi_0^*=-i(\hat{A},\omega)\in\op{End}(\La^{n,0}V^*).
 $$
 \end{lem}

 \begin{proof}
In a unitary (holomorphic) basis $\{e_a\}$ of $V$ with the dual basis $\{e^a\}$ of $V^*$ we get
$\omega=i(e^1\we e^{\bar1}+{\dots}+e^n\we e^{\bar n})\in\La^{1,1}V^*$. The holomorphic volume form is
$\Omega=e^1\we\dots\we e^n\in\La^{n,0}V^*$. Hence
$\pi_0A^*\pi_0^*:\La^{n,0}V^*\to\La^{n,0}V^*$ is equal to
 $$
\langle A^*\Omega,\Omega\rangle=
\langle A^*e^1,e^1\rangle+{\dots}+\langle A^*e^n,e^n\rangle
=(\hat{A},e^1\we e^{\bar1}+{\dots}+e^n\we e^{\bar n})
 $$
and the last expression is $(\hat{A},-i\omega)$.
 \end{proof}

For $R\in\La^2V^*\ot\op{End}(V)$ with values in anti-symmetric endomorphisms
define $\tilde{R}\in\op{End}(\La^2V^*)$ as the composition (where $\flat=\sharp^{-1}$)
 $$
\La^2V^*\ot\op{End}_{\text{skew}}(V)\stackrel{\flat}\to
\La^2V^*\ot\La^2V^*\stackrel{\1\ot\sharp^{\we2}}\longrightarrow\La^2V^*\ot\La^2V=\op{End}(\La^2V^*).
 $$
In other words, if $R=\sum\a_k\ot A_k$, then for $\b\in\La^2V^*$ the action is
$\tilde{R}(\b)=\sum(\hat{A}_k,\b)\a_k$. Now the previous lemma implies
 \begin{cor}
Denote by $R_\nabla$ the curvature of the connection induced from the Levi-Civita connection
on $\La^nTM^*$. Then
 $$
R_\nabla|_K=i\tilde{R}(\omega)\in\Omega^2(M).
 $$
 \end{cor}

 \begin{proof}
If $R=\sum\a_k\ot A_k$, then $R_\nabla=-\sum\a_k\ot A_k^*$ and therefore
$\pi_0R_\nabla^*\pi_0^*=i\sum(\hat{A}_k,\omega)\a_k=i\tilde{R}(\omega)$.
 \end{proof}

 This corollary and decomposition \eqref{SFF} yield formula \eqref{NF} from the introduction.

\subsection{Type of the second fundamental form}\label{S32}
Recall \cite{GH} that for the Chern connection the second fundamental form is always of type $(1,0)$.
For the Levi-Civita connection this is not always so, however for the subbsundle
$K\subset\La^nT^*M\ot\C$ this property holds.
 \begin{lem}
The second fundamental form of the canonical bundle $K$ satisfies
$\Phi\in\Omega^{1,0}(M)\ot\Gamma(\op{Hom}(K,\La^nT^*M\ot\C/K))$.
 \end{lem}

 \begin{proof}
Let us first prove the same property for the subbundle $\La^{1,0}(M)\subset T^*M\ot\C$.
Let $e_a$ be a (local) unitary frame and $e^a$ the dual co-frame.
Since $\nabla e^a=-\vartheta_c^ae^c-\vartheta_{\bar c}^ae^{\bar c}$, the claim means
$\vartheta_{\bar c}^a\in\Omega^{1,0}(M)$.

Decompose the $(0,1)$-part of this connection form
component: $\vartheta''{}_{\bar b}^a=\sum\beta_{abc}e^{\bar c}$.
The Leibnitz rule applied to $d(e_a,e_b)=0$ implies that $\beta_{abc}$ is skew-symmetric in $ab$.

On the other hand, $\nabla$ is torsion-free and so
 $$
de^a=\op{alt}[\nabla e^a]=-\sum(\vartheta^a_c\we e^c+\vartheta_{\bar c}^a\we e^{\bar c}).
 $$
The Nijenhuis tensor of $J$ vanishes $0=d^{-1,2}:\Omega^{1,0}(M)\to\Omega^{0,2}(M)$
and this implies $\sum\vartheta''{}_{\bar c}^a\we e^{\bar c}=0$. Therefore
$\beta_{abc}$ is symmetric in $bc$, and the $S3$-lemma
$(V\wedge V\otimes V)\cap(V\otimes V\odot V)=0$
 yields $\beta_{abc}=0$, i.e.\ $\vartheta''{}_{\bar c}^a=0$.

Now we pass to the subbundle $K=\La^{n,0}(M)\subset\La^nT^*M\ot\C$. Since
 $$
\nabla(e^1\we\dots\we e^n)=-\sum\vartheta^a_{\bar c}\ot(e^1\we\dots\we e^{a-1}\we
e^{\bar c}\we e^{a+1}\we\dots\we e^n)\,\op{mod}K
 $$
the claim follows.
 \end{proof}

 \begin{cor}
We have: $-i\Phi^*\we\Phi\ge0$.
 \end{cor}

 \begin{proof}
Since $\Phi$ has type $(1,0)$ we conclude
 $$
\tfrac1i(-i\Phi^*\we\Phi)(X',X'')=\Phi^*(X'')\Phi(X')=(\Phi(X'))^*\Phi(X')\ge0.
 $$

\vspace{-19pt}
 \end{proof}

%%%%%%%%%%%%%%%%%%%%%%%%%%%%%%%%%%%%%%%%%%%%%%%%%%%%%%%%%%%%%%%%%%%%%%%%%%
\section{Proof of Theorem \ref{thm2}}\label{S4}

This section contains the proof of Theorem \ref{thm2}, following the approach of \cite{BHL}.
Theorem \ref{thm1} is an immediate corollary.

We first prove a quantitative assertion that a small perturbation of a positive form is positive.
Let $(g,J,\omega)$ be a (linear) Hermitian structure on a vector space $V$,
$n=\dim_\C V=\tfrac12\dim_\R V$. The Euclidean structure on $V$ induces the following
norm on $\op{End}(V)$: $\|A\|_E=\sqrt{\sum|A_i^j|^2}$, where $A=[A_i^j]$ is the matrix representation
in some unitary basis. It also yields the norm $\z\mapsto\|\z\|_{\La^2}$ on $\La^2V^*$.

Note that the embedding $\La^2V^*\simeq\op{End}_\text{skew}(V)\subset\op{End}(V)$,
$\hat{A}\mapsto A$, scales the norm: $\|\z\|^2_E=2\|\z\|_{\La^2}^2$.
Below we identify $\hat{A}$ with $A$.

 \begin{lem}\label{LL}
Let $\z_0$ be a real $(1,1)$-form and $\z_0\geq\omega$. Then any real $(1,1)$-form
$\z$, such that $\|\z-\z_0\|_{\La^2}\leq\frac1{2\sqrt{n}}$, is nondegenerate.
 \end{lem}

 \begin{proof}
Recall that if $\|A\|_E<1$ for $A\in\op{End}(V)$, then $\1-A\in\op{End}(V)$ is invertible. Indeed,
$(\1-A)^{-1}=\sum_{k=0}^\infty A^k$.

Diagonalize $\oo$ and $\z$ simultaneously: in some unitary co-frame $e^a$
 $$
\oo=i\sum e^a\we e^{\bar a},\qquad \z_0= i\sum \l_a e^a\we e^{\bar a},
 $$
and $\l_a\geq1$ by the assumptions. Then $\z_0^{-1}=-i\sum \l_a^{-1} e^a\we e^{\bar a}$
and $\|\z_0^{-1}\|_E^2=2\sum\l_a^{-2}\leq 2n$.

Decompose $\z=\z_0+(\z-\z_0)=\z_0\cdot(\1+\z_0^{-1}(\z-\z_0))$. The claim follows from
$\|\z_0^{-1}(\z-\z_0)\|_E\leq\|\z_0^{-1}\|_E\cdot \|\z-\z_0\|_E<
\sqrt{2n}\cdot\frac{\sqrt{2}}{2\sqrt{n}}=1$.
 \end{proof}

Now the proof of Theorem \ref{thm2} is concluded as follows. Let $n=3$.
Normalize $g$ by the requirement $\op{Sp}(\tilde{R})\in(\frac56,\frac76)$.
Then $\op{Sp}(\1-\tilde{R})\in(-\frac16,\frac16)$,
and so $\|\tilde{R}(\omega)-\omega\|<\frac16\|\omega\|=\frac16\sqrt{3}=\frac1{2\sqrt{3}}$.
By \eqref{NF} we get
 $$
-i\Omega=\tilde{R}(\omega)-i\Phi^*\we\Phi=(\tilde{R}(\omega)-\omega)+(\oo-\Phi^*\we\Phi).
 $$
Since $\z_0=\oo-\Phi^*\we\Phi\ge\oo$, then by Lemma \ref{LL} we conclude that $i\Omega$,
and hence $\Omega$ are nondegenerate. Since $\Omega$ is closed by Bianchi's identity,
it is symplectic on $\Ss$, which is a contradiction. \qed

%%%%%%%%%%%%%%%%%%%%%%%%%%%%%%%%%%%%%%%%%%%%%%%%%%%%%%%%%%%%%%%%%%%%%%%%%%
\section{Another approach}\label{S5}

In this section we give yet another proof of Theorem \ref{thm1} due to K. Sekigawa and
L. Vanhecke \cite{SV}. We should warn the reader of some unspecified sign choices
 % inconsistencies for holomorphic Ricci curvature and the Chern class
in their paper, which we amend here.

Our sign conventions in this respect are in agreement with \cite{Gr,GBNV},
though in these sources the curvature is defined as minus that of ours.
Since there are several differences in sign agreements, for instance in passing from $(g,J)$ to
$\omega$, in Ricci contraction etc, this will be reflected in sign differences of our formulae,
which otherwise are fully equivalent.
 % Also the relation of $\omega$ to $g,J$ has the sign difference.

\subsection{The first Chern class}
Given a connection $D$ and its curvature tensor $R_D$ on an almost Hermitian manifold $(M,g,J)$ of dimension $2n$
define its holomorphic Ricci curvature by
 $$
\op{Ric}^*_D(X,Y)=-\op{Tr}\Bigl(R(X,J\cdot)JY\Bigr)=\frac12\sum_{i=1}^{2n} R_D(X,JY,e_i,Je_i),
 $$
where $e_1,e_2=Je_1,\dots,e_{2n-1},e_{2n}=Je_{2n-1}$ is a $J$-adapted orthonormal basis.
For the characteristic (Chern) connection $\mathbb{D}$ the 2-form
 $$
\gamma_1(X,Y)=\frac{-1}{2\pi}\op{Ric}^*_{\mathbb{D}}(X,JY)=\frac1{2\pi}\op{Ric}^*_{\mathbb{D}}(JX,Y)
 $$
represents the first Chern class $c_1=[\gamma_1]$, see \cite{GBNV}.
When passing to the Levi-Civita connection $\nabla$, this simple formula is modified.

A relation between the two connections is given by \cite[(6.2)]{GBNV}
that, in the case of integrable $J$, states
 $$
g(\mathbb{D}_XY,Z)= g(\D_XY,Z)-\tfrac12g(JX,\nabla_Y(J)Z-\nabla_Z(J)Y),
 $$
with the canonical connection $\D$ given by (\ref{can}). This allows to express
the first Chern form in terms of $\nabla$ (the curvature of $\D$ is expressed through that
of $\nabla$ in \cite{T}, and the curvature of $\mathbb{D}$ -- in \cite{GBNV}).

Define the 2-forms $\psi(X,Y)=-2\op{Ric}^*_{\nabla}(X,JY)=\sum R_\nabla(X,Y,e_i,Je_i)$ and
$\vp(X,Y)=\op{Tr}\Bigl(J(\nabla_XJ)(\nabla_YJ)\Bigr)=-\sum(\nabla_XJ)^a_b(\nabla_{JY}J)^b_a$.
With these choices (cf.\ \cite{GBNV,SV}) the first Chern form is given by
 \begin{equation}\label{Chern}
8\pi\gamma_1=2\psi+\vp.
 \end{equation}

 \subsection{Alternative proof of Theorem \ref{thm1} }
Now suppose that $g$ has constant sectional curvature $k>0$, i.e.\
 $$
R_\nabla(X,Y,Z,T)=g(R_\nabla(X,Y)Z,T)=k\cdot(g\varowedge g)(X,Y,Z,T),
 $$
where $(g\varowedge g)(X,Y,Z,T)=g(X,Z)g(Y,T)-g(X,T)g(Y,Z)$ is the Kulkarni-Nomizu product, whence
 \begin{gather*}
\op{Ric}_\nabla(X,Y)=\sum R_\nabla(X,e_i,Y,e_i)=(2n-1)\,k\,g(X,Y),\\
\op{Ric}^*_\nabla(X,Y)=\sum R_\nabla(X,e_i,JY,Je_i)=k\,g(X,Y).
 \end{gather*}
In other words, this metric $g$ is both Einstein and $*$-Einstein, and the scalar and $*$-scalar
curvatures are both positive.

Thus both $\op{Ric}_\nabla$ and $\op{Ric}^*_\nabla$ are positive definite, and hence $\psi>0$.
Now since $\vp(X,JX)=\|\nabla_{JX}J\|^2=\|\nabla_XJ\|^2\ge0$, we have $\vp\ge0$, and
consequently $\gamma_1>0$. Integrating $\gamma_1^n$ yields $c_1^n(M)\neq0$.

Returning to the case $M=\Ss$, $n=3$, and the standard round metric $g=g_0$
of constant sectional curvature 1,
we obtain a contradiction because $c_1\in H^2(\Ss)=0$,
and so $c_1^3=0$ as well. \qed

%%%%%%%%%%%%%%%%%%%%%%%%%%%%%%%%%%%%%%%%%%%%%%%%%%%%%%%%%%%%%%%%%%%%%%%%%%
\section{Generalization of the idea of Section \ref{S5}}\label{S6}

If we perturb the metric $g$ starting from $g_0$, it is no longer $*$-Einstein, and
the argument of the previous section literally fails.

However, since the space of $g$-orthogonal complex structures
$\mathcal{J}_g\simeq O(2n)/U(n)$
is compact, the image of the map
 $$
\mathcal{J}_g\ni J\mapsto\op{sym}[\psi(\cdot,J\cdot)]\in\Gamma(\odot^2T^*M)
 $$
is close to the  one-point set $\{g_0\}$ (because $g_0$ is $*$-Einstein)
and so is positive for $g$ sufficiently close to $g_0$. Thus we still get
the inequality $8\pi\gamma_1=2\psi+\vp>0$ as in the previous section, and so
conclude non-existence of $g$-orthogonal complex structures $J$ on $\Ss$ for an open set
of metrics $g\in\Gamma(\odot^2_+T^*\Ss)$ in $C^2$-topology.

A quantitative version of this idea is a novel result given below.

\subsection{Bounds in the space of curvature tensors}
Fix a Euclidean space $V$ of even dimension $2n$ with metric $g=\langle\cdot,\cdot\rangle$,
and consider the space $\mathcal{R}$ of algebraic curvature tensors on it.
Identifying $(3,1)$ and $(4,0)$ tensors via the metric,
$\mathcal{R}=\op{Ker}[\wedge:\odot^2\Lambda^2V^*\to\Lambda^4V^*]$.

In this subsection we restrict to linear tensors in $V$. Denote by $\mathcal{P}$ the space
$\{R\in\mathcal{R}:\op{Ric}^*_R(X,X)\ge0\ \forall X\in V,\forall J\in\mathcal{J}_g\}$,
where $\op{Ric}^*_R$ is computed via $R$ and $J$ as in the previous section.

This can be exposed in index terms as follows.
Denote by $\mathcal{F}_g$ the space of $g$-orthonormal frames $e=\{e_1,\dots,e_{2n}\}$ on $V$.
Each such frame yields an orthogonal complex structure on $V$ by
$Je_i=(-1)^{i-1}e_{i^\#}$, where $i^\#=i-(-1)^i$. For every $e\in\mathcal{F}_g$ and $R\in\mathcal{R}$
compute $\alpha_{ij}=\sum_{k=1}^{2n}R(e_i,e_k,e_{j^\#},e_{k^\#})=(-1)^{i+j}\alpha_{j^\#i^\#}$
and form the symmetric matrix $A$ with entries $a_{ij}=\frac12(\alpha_{ij}+\alpha_{ji})$.
Then $R\in\mathcal{P}$ iff $A$ is positive semidefinite for every $e\in\mathcal{F}_g$,
and this can be determined by finite-dimensional optimization via the Silvester criterion.

A simple sufficient criterion for this is the following. Introduce the following $L^\infty$-norm
on $\mathcal{R}$: $\|R\|_\infty=\max_{\{|v_i|=1\}}|R(v_1,v_2,v_3,v_4)|$.
 \begin{lem}
If $\|R-g\varowedge g\|_\infty\leq\frac1{2n}$, then $R\in\mathcal{P}$.
 \end{lem}

 \begin{proof}
Denote $\check{R}= R-g\varowedge g$. Then for any $J\in\mathcal{J}_g$ and $X\in V$ with $\|X\|=1$
we get
 \begin{multline*}
\op{Ric}^*_R(X,X)=\sum R(X,e_i,JX,Je_i)=\|X\|^2+\sum\check{R}(X,e_i,JX,Je_i)\\
\ge\|X\|^2-2n\|X\|^2\max_{\|u\|=\|v\|=1}|\check{R}(u,v,Ju,Jv)|\ge0.
 \end{multline*}
Thus $R\in\mathcal{P}$.
 \end{proof}

\subsection{A non-existence alternative to Theorem \ref{thm2}}
Write $g\in\mathcal{P}$ if the curvature tensor of $g$ satisfies this positivity property on every
tangent space $V=T_xM$, $x\in M$.
The set $\mathcal{P}$ is a neighborhood of
the round metric $g_0$ on $M=\Ss$ in the space of all metrics in $C^2$-topology.

 \begin{theorem}\label{Thlast}
$\Ss$ possesses no Hermitian structure $(g,J,\omega)$ with $g\in\mathcal{P}$.
 \end{theorem}

 \begin{proof}
In formula \eqref{Chern} $\vp\ge0$, and if $g\in\mathcal{P}$ then $\psi\ge0$ as well.
Thus $\gamma_1\ge0$ and we conclude $0=c_1^3[\Ss]=\int_{\Ss}\gamma_1^3\ge0$.
This integral is the sum of several non-negative summands, the last of which is $\int_{\Ss}\vp^3$.
Since all of these summands have to vanish, we conclude $\vp=0$ implying $\|\nabla J\|^2=0$.
Thus $\nabla J=0$ meaning that $(g,J,\omega)$ is a K\"ahler structure on $\Ss$ and this is a
contradiction.
  \end{proof}

 \begin{cor}
If $\|R_\nabla-g\varowedge g\|_\infty\leq\frac16$ for the curvature $R_\nabla$ of a metric $g$ on $\Ss$,
then no $g$-orthogonal almost complex structure $J$ is integrable.
 \end{cor}

Note that this can be again considered as a perturbation result for the metric $g_0$,
for which the curvature tensor is $R_{\nabla_0}=g_0\varowedge g_0$:
If $\|R_\nabla-R_{\nabla_0}\|_\infty\leq\epsilon_1$, $\|g-g_0\|_\infty\leq\epsilon_2$ and
$\epsilon_1+4\epsilon_2+2\epsilon_2^2\leq\frac16$ (one can check that this follows from the linear constraint
$\epsilon_1+\bigl(2+\sqrt{13/3}\bigr)\epsilon_2\leq\frac16$),
 % Indeed: the linear constraint implies $\epsilon_2\leq\frac12\sqrt{13/3}-1$, i.e. $2+\epsilon_2\leq\frac12\sqrt{13/3}+1$,
 % so $\epsilon_1+2\epsilon_2(2+\epsilon_2)\leq\epsilon_1+(2+\sqrt{13/3}\epsilon_2\leq\frac16$.
the claim follows. Indeed (note that $\|v\|^2=g_0(v,v)$) we have:
 \begin{gather*}
\tfrac12\|g\varowedge g-g_0\varowedge g_0\|_\infty\leq
\max_{\|v_i\|=1}|g(v_1,v_3)g(v_2,v_4)-g_0(v_1,v_3)g_0(v_2,v_4)|\leq\\
\max_{\|v_i\|=1}|g(v_1,v_3)-g_0(v_1,v_3)|\,|g(v_2,v_4)|+|g(v_2,v_4)-g_0(v_2,v_4)|\,|g_0(v_1,v_3)|\\
\leq\|g-g_0\|_{\infty}(2+\|g-g_0\|_{\infty}).
 \end{gather*}
Thus $\|R_\nabla-g\varowedge g\|_\infty\leq\|R_\nabla-g_0\varowedge g_0\|_\infty+\|g_0\varowedge g_0-g\varowedge g\|_\infty\leq
\epsilon_1+2\epsilon_2(2+\epsilon_2)$.

%%%%%%%%%%%%%%%%%%%%%%%%%%%%%%%%%%%%%%%%%%%%%%%%%%%%%%%%%%%%%%%%%%%%%%%%%%
\section{Concluding remarks}\label{S7}

The non-existence results of this paper are not sharp. Indeed, both Theorems \ref{thm2} and
Corollary of Theorem \ref{Thlast} deal with rough upper bound and could be further improved.

Note also that the property of an almost complex structure $J$ being $g$-orthogonal depends
only on the conformal class of the metric\footnote{This has the following corollary generalizing Theorem \ref{thm1}:
If $g$ is conformally equivalent to $g_0$ on $\Ss$, then the space $J\in\mathcal{J}_g(\Ss)$ contains no complex structures.}, while the Riemann and holomorphic Ricci tensors, used in
the proofs, are not conformally invariant. It is a challenge to further elaborate the
results to get better bounds, implying non-existence of orthogonal complex structures
in a larger neighborhood of the round metric on $\Ss$.

Every almost complex structure $J$ on $\Ss$ is orthogonal with respect to some metric $g$, but
as this (or any conformally equivalent metric) can be far from $g_0$, the positivity argument will not work.

Note that all known proofs of non-existence of Hermitian structures for certain $g$ use only one property of the 6-sphere,
namely that $H^2(\Ss)=0$. It would be interesting to find a proof of non-existence of orthogonal complex
structures based on some other ideas.
 % that arrive to a contradiction by some other means.

\bigskip

\textsc{Acknowledgement.} I am grateful to Oliver Goertsches for a careful reading of the first draft,
and useful suggestions on the exposition.

%%%%%%%%%%%%%%%%%%%%%%%%%%%%%%%%%%%%%%%%%%%%%%%%%%%%%%%%%%%%%%%%%%%%%%%%%%


\begin{thebibliography}{11}
\footnotesize

\bibitem{B}
A. Blanchard, {\it Recherche de structures analytiques complexes sur certaines vari\'et\'es\/},
C.\,R.\,Acad.\ Sci., Paris {\bf 236}, 657-659 (1953)

\bibitem{BHL}
G. Bor, L. Hern\'andez-Lamoneda, {\it The canonical bundle of a Hermitian manifold},
Bol. Soc. Mat. Mexicana (3), vol. {\bf 5}, 187-198 (1999)

\bibitem{Gr}
A. Gray, {\it The Structure of Nearly K\"ahler manifolds}, Math.Ann. {\bf 233}, 233-248 (1976)

\bibitem{GBNV}
A. Gray, M. Baros, A. Naviera, L. Vanhecke,
{\it The Chern numbers of holomorphic vector bundles and formally holomorphic connections
of complex vector bundles over almost complex manifolds}, J. Reine Angew. Math. {\bf 314}, 84-98 (1980)

\bibitem{GrH}
A. Gray, L.\,M. Hervella, {\it The sixteen classes of almost Hermitian manifolds
and their linear invariants\/}, Ann. Mat. Pura Appl. (4) {\bf 123}, 35-58 (1980)

\bibitem{GH}
P. Griffiths, J. Harris, {\it Principles of algebraic geometry}, John Wiley \& Sons (1994)

\bibitem{LeB}
C. LeBrun, {\it Orthogonal complex structures on $\Ss$}, Proc. Am. Math. Soc. {\bf 101}, 136-138 (1987)

\bibitem{Sal}
S. Salamon, {\it Orthogonal Complex Structures}, Differential Geometry and its Applications,
Proc. Conf., Aug. 28 - Sept. 1, 1995, Brno, Czech Republic Masaryk University, Brno, 103-117 (1996)

\bibitem{SV}
K. Sekigawa, L. Vanhecke, {\it Almost Hermitian manifolds with vanishing first Chern classes},
Rend. Sem. Mat. Univer. Politec. Torino {\bf50}, 195-208 (1992)

\bibitem{T}
S. Tachibana, {\it On automorphisms of certain compact almost Hermitian spaces},
T\^{o}hoku Math. J. {\bf 13}, 179-185 (1961)

% \bibitem{W}
% C.\,M. Wood, {\it Harmonic almost-complex structures}, Compositio Mathematica
% {\bf 99} (1995), 183-212.

\end{thebibliography}
\end{document}